\documentclass[11pt,a4paper]{article}
\usepackage{amsmath}
\usepackage{amsthm}
\usepackage{makeidx}
\usepackage{latexsym}
\usepackage{graphicx}
\usepackage{amssymb}
\usepackage{hyperref}
\usepackage{float}
\usepackage[latin1]{inputenc}
\usepackage{amssymb}

\date{}

\theoremstyle{plain}
 \newtheorem{theorem}{Theorem}[section]

\theoremstyle{definition}
 \newtheorem{remark}[theorem]{Remark}
 
\theoremstyle{definition}
 \newtheorem{example}[theorem]{Example} 

\newtheorem{definition}{Definition}

\begin{document}
%\begin{frontmatter}
\title{On the Hermite problem for cubic irrationalities}
%\author{Nadir}{Murru}{N.~Murru}{Turin}
%\contact{Department of Mathematics,University of Turin, Via Carlo Alberto 10, Turin, 10123, ITALY}{nadir.murru@unito.it}
\author{Nadir Murru\\
Department of Mathematics, University of Turin\\
Via Carlo Alberto 10, Turin, 10123, ITALY\\
nadir.murru@unito.it}
%\address{Department of Mathematics,University of Turin\\
%Via Carlo Alberto 10, Turin, 10123, ITALY\\
%nadir.murru@unito.it}

\maketitle

\begin{abstract}

In this paper, the Hermite problem has been approached finding a periodic representation (by means of periodic rational or integer sequences) for any cubic irrationality. In other words, the problem of writing cubic irrationals as a periodic sequence of rational or integer numbers has been solved. In particular, a periodic multidimensional continued fraction (with pre--period of length 2 and period of length 3) is proved convergent to a given cubic irrationality, by using the algebraic properties of cubic irrationalities and linear recurrent sequences. This multidimensional continued fraction is derived from a modification of the Jacobi algorithm, which is proved periodic if and only if the inputs are cubic irrationals. Moreover, this representation provides simultaneous rational approximations for cubic irrationals.

\end{abstract}
%\begin{keyword}
%cubic irrationals, diophantine approximation, generalized continued fractions, Hermite problem, Jacobi--Perron algorithm, linear recurrent sequences  \\
%\emph{Mathematics Subject Classification 2010}: 11J70, 11J68.
%\end{keyword}

%\end{frontmatter}

\section{Introduction}

\indent \indent In 1839, Hermite (see the letters published in 1850, \cite{Her}) posed to Jacobi the problem of generalizing the construction of continued fractions to higher dimensions. In particular, he asked for a method of representing algebraic irrationalities by means of periodic sequences that can highlight algebraic properties and possibly provide rational approximations. Hermite especially focused the attention on cubic irrationalities. \\
\indent Continued fractions completely solve this problem for quadratic irrationalities, but the problem for algebraic numbers of degree $\geq$ 3 is still open. In 1868, Jacobi (see the VI volume of the book Ges. Werke published in 1891, \cite{Jac}) built an algorithm that produces a generalization of the classic continued fractions. These multidimensional continued fractions, if periodic, converge to cubic irrationalities, but the vice versa has been never proved. In 1907, Perron \cite{Per} developed a generalization of this algorithm for any algebraic irrationalities (for a complete survey about the Jacobi--Perron algorithm see \cite{Ber} and \cite{Sc}). \\
\indent The study of periodic representations for algebraic numbers has interested many mathematicians. This is a very beautiful theoretical (and not only) question: can irrational numbers have a periodic representation? During the years many attempts have been performed. Several kinds of multidimensional continued fractions have been deeply studied in various works.\\  
\indent In \cite{Ber2}, \cite{Ber3}, \cite{Ber4}, Bernstein studied and proved the convergence of the Jacobi--Perron algorithm for a vast class of algebraic irrationals. Further results on the Jacobi--Perron algorithm can be found in \cite{Raju} and \cite{Vau}. Many different algorithms are summarized and compared in \cite{Bre} and \cite{Bry}. Moreover, see the beautiful book of Schweiger \cite{Scb} for a guide about multidimensional continued fractions.\\
\indent  The usual approach to the Hermite problem contemplates the research of functions whose iteration on algebraic irrationalities yields a periodical algorithm. In this sense, Tamura and Yasutomi \cite{Ta}, \cite{Ta2} recently presented a modified Jacobi--Perron algorithm. Similarly, the Jacobi--Perron algorithm has been modified using different functions, e.g., in \cite{Hendy}, \cite{Willi}, \cite{Martin}, \cite{Sc2}.\\
\indent A very interesting approach can be found in the works \cite{gar1}, \cite{gar3}, \cite{gar2}, where a multidimensional continued fraction related to triangle sequences is studied. Moreover, in \cite{Be} a generalization of the Minkowski question--mark function is developed. Finally, a completely different approach to multidimensional continued fractions can be found in \cite{Kar}. \\
\indent All these algorithms, when periodic, provide sequences that converge to cubic irrationalities. However, none algorithm has been proved to become periodic when the input is a cubic irrational. Thus, not exists any algorithm that provides a periodic representation for a given cubic irrationality.\\
\indent In this paper, the Hermite problem has been approached finding a periodic representation for an irrational number satisfying the cubic equation $x^3-px^2-qx-r=0$, with $p,q,r$ rational numbers. In other words, the problem of writing cubic irrationals as a periodic sequence of rational or integer numbers has been solved. The Hermite problem can not be claimed completely solved, since the periodic representation does not derive from an algorithm defined over all real numbers. However, the results obtained are important, since for the first time we have found a periodic representation for any cubic irrational. \\
\indent The periodic representation has been directly found by means of elementary techniques that only involve the algebraic properties of cubic irrationalities and the properties of linear recurrent sequences. Moreover, an algorithm, which provides a periodic sequence when a cubic irrational is given in input, is derived.\\
\indent In section 2, the fundamental properties of the multidimensional continued fraction (named ternary continued fraction), derived from the Jacobi algorithm, are presented. In section 3, the main case ($\alpha$ root largest in modulus of $x^3-px^2-qx-r$) is treated. A periodic expansion that converges to the couple $(\cfrac{r}{\alpha},\alpha)$ is shown. In section 4, all the remaining cases of cubic irrationalities satisfying a cubic equation $x^3-px^2-qx-r=0$ are treated. In this way, we can say that a real number $\alpha$ is a cubic irrational $\Leftrightarrow$ $\alpha$ can be represented by means of a periodic ternary continued fraction (that is convergent).\\
\indent The iteration on cubic irrationalities  of the map in section 5 provides (according to the Jacobi algorithm) the periodic expansion found in the previous sections. Section 6 is devoted to the conclusions.

\section{Ternary continued fractions}
The Jacobi algorithm associates a couple of integer sequences to a couple of real numbers by the following equations:
\begin{equation} \label{algobiforcanti} \begin{cases} a_n=[\alpha_n] \cr b_n=[\beta_n] \cr \alpha_{n+1}=\cfrac{1}{\beta_n-[\beta_n]} \cr \beta_{n+1}=\cfrac{\alpha_n-[\alpha_n]}{\beta_n-[\beta_n]}  \end{cases}, \end{equation}
$n=0,1,2,...$, for any couple of real numbers $\alpha=\alpha_0$ and $\beta=\beta_0$. It follows that 
$$ \begin{cases} \alpha_n=a_n+\cfrac{\beta_{n+1}}{\alpha_{n+1}}  \cr \beta_n=b_n+\cfrac{1}{\alpha_{n+1}}  \end{cases} $$
Therefore, the real numbers $\alpha$ and $\beta$ are represented by the sequences $(a_n)_{n=0}^\infty$, $(b_n)_{n=0}^\infty$ as follows:
\begin{equation} \label{fcb} \alpha=a_0+\cfrac{b_1+\cfrac{1}{a_2+\cfrac{b_3+\cfrac{1}{\ddots}}{a_3+\cfrac{\ddots}{\ddots}}}}{a_1+\cfrac{b_2+\cfrac{1}{a_3+\cfrac{\ddots}{\ddots}}}{a_2+\cfrac{b_3+\cfrac{1}{\ddots}}{a_3+\cfrac{\ddots}{\ddots}}}} \quad \text{and} \quad \beta=b_0+\cfrac{1}{a_1+\cfrac{b_2+\cfrac{1}{a_3+\cfrac{\ddots}{\ddots}}}{a_2+\cfrac{b_3+\cfrac{1}{\ddots}}{a_3+\cfrac{\ddots}{\ddots}}}}\end{equation}
We call \emph{ternary} continued fraction (as named, e.g., in \cite{Daus} and \cite{Le}, where these objects have been studied independently from the generating algorithm) such a couple of objects representing the numbers $\alpha$ and $\beta$ and we write 
\begin{equation} \label{fcb2} (\alpha,\beta)=[\{a_0,a_1,a_2,...\},\{b_0,b_1,b_2,...\}], \end{equation}
where $a_i$'s and $b_i$'s are called \emph{partial quotients}.
\begin{remark}
Ternary continued fraction are also called bifurcating continued fraction as in \cite{abcm} and \cite{gupta}.
\end{remark}
Similarly to classical continued fractions, the notion of convergent is introduced as follows (for a complete survey of the Jacobi--Perron algorithm see \cite{Ber}):
$$[\{a_0,a_1,...,a_n\},\{b_0,b_1,...,b_n\}]=(\cfrac{A_n}{C_n},\cfrac{B_n}{C_n}),\quad \forall n\geq0$$ 
$$\lim_{n\rightarrow\infty}\cfrac{A_n}{C_n}=\alpha,\quad \lim_{n\rightarrow\infty}\cfrac{B_n}{C_n}=\beta$$
are the $n$--\emph{th convergents} of the ternary continued fraction \eqref{fcb2}, where $A_n,B_n,C_n$ satisfy the following recurrent relations 
\begin{equation} \label{convergenti} \begin{cases} A_n=a_nA_{n-1}+b_nA_{n-2}+A_{n-3} \cr B_n=a_nB_{n-1}+b_nB_{n-2}+B_{n-3} \cr C_n=a_nC_{n-1}+b_nC_{n-2}+C_{n-3}  \end{cases},\quad \forall n\geq 1 \end{equation} 
with initial conditions
$$\begin{cases}  A_{-2}=1,\quad A_{-1}=0, \quad A_0=a_0 \cr B_{-2}=0,\quad B_{-1}=1, \quad B_0=b_0 \cr C_{-2}=0,\quad C_{-1}=0, \quad C_0=1 \end{cases}$$
Furthermore, a matricial approach is known. Indeed, it is easy to prove by induction that
\begin{equation} \label{matrix-fcb} \begin{pmatrix}  a_0 & 1 & 0 \cr b_0 & 0 & 1 \cr 1 & 0 & 0 \end{pmatrix}... \begin{pmatrix} a_n & 1 & 0 \cr b_n & 0 & 1 \cr 1 & 0 & 0  \end{pmatrix}=\begin{pmatrix} A_n & A_{n-1} & A_{n-2} \cr B_n & B_{n-1} & B_{n-2} \cr C_n & C_{n-1} & C_{n-2}  \end{pmatrix}\end{equation} 
\[ \Updownarrow \] \[(\cfrac{A_n}{C_n},\cfrac{B_n}{C_n})=[\{a_0,...,a_n\},\{b_0,...,b_n\}], \]
for $n=0,1,2,...$.
\begin{remark}
A ternary continued fraction \eqref{fcb} can converge to a couple of real numbers although the partial quotients are not obtained by the Jacobi algorithm. Thus, it is possible to study convergence of ternary continued fractions independently from the origin of the partial quotients. In sections 3 and 4, we find the partial quotients $a_i$'s and $b_i$'s such that the ternary continued fraction \eqref{fcb} converges to a given cubic irrational. In Section 5, these partial quotients are obtained by the Jacobi algorithm \eqref{algobiforcanti} by means of two functions $f_z^\alpha,g_z^\alpha$ that substitute the role of the floor function.
\end{remark}
In \cite{abcm}, the authors studied the convergence of ternary continued fractions with rational partial quotients, finding infinitely many periodic representations for every cubic root.
\begin{theorem} \label{fcb-teo1} The periodic ternary continued fraction
\begin{equation} \label {fcbredei} [\{z,\cfrac{2z}{d},\overline{\cfrac{3dz}{z^3+d^2},3z,\cfrac{3z}{d}}\},\{0,-\cfrac{z^2}{d},\overline{-\cfrac{3z^2}{z^3+d^2},-\cfrac{3dz^2}{z^3+d^2},-\cfrac{3z^2}{d}}\}]  \end{equation}
converges for every integer $z\not=0$ to the couple of irrationals $(\sqrt[3]{d^2},\sqrt[3]{d})$, for $d$ integer not cube.
\end{theorem}
\begin{remark}
A ternary continued fraction with rational partial quotients is clearly determined by sequences of integer numbers. Indeed, given
$$[\{\cfrac{a_0}{b_0},\cfrac{a_1}{b_1},...\},\{\cfrac{c_0}{d_0},\cfrac{c_1}{d_1},...\}],$$
where $a_i, b_i, c_i, d_i$ integer numbers, for $i=0,1,2,...$, then
\begin{equation} \label{rm} \begin{pmatrix} a_0d_0 & b_0d_0 & 0 \cr c_0b_0 & 0 & b_0d_0 \cr b_0d_0 & 0 & 0 \end{pmatrix}\cdots \begin{pmatrix} a_nd_n & b_nd_n & 0 \cr c_nb_n & 0 & b_nd_n \cr b_nd_n & 0 & 0 \end{pmatrix}=  \end{equation} 
$$\begin{pmatrix}  d_0s_n & d_0b_nd_n s_{n-1} & d_0b_nd_nb_{n-1}d_{n-1} s_{n-2} \cr b_0d_1s'_n & b_0b_nd_ns'_{n-1} & b_0b_nd_nb_{n-1}d_{n-1}s'_{n-2} \cr b_0d_0d_1s''_n & b_0d_0d_1b_nd_ns''_{n-1} & b_0d_0d_1b_nd_nb_{n-1}d_{n-1}s''_{n-2} \end{pmatrix} $$
$$ \Updownarrow $$
$$ [\{\cfrac{a_0}{b_0},...,\cfrac{a_n}{b_n}\},\{\cfrac{c_0}{d_0},...,\cfrac{c_n}{d_n}\}]=(\cfrac{s_n}{b_0d_1s''_n},\cfrac{s'_n}{d_0s''_n})=(\cfrac{A_n}{C_n},\cfrac{B_n}{C_n}), $$
where  
$$\begin{cases} s_0=a_0,\quad s_1=a_0a_1d_1+b_0b_1c_1,\quad s_2=a_2d_2s_{1}+b_2b_{1}c_2d_{1}s_{0}+b_2b_{1}b_{0}d_2d_{1} \cr s'_0=c_0,\quad s'_1=a_1c_0+b_1d_0, \quad s'_2=b_1b_2c_0c_2+a_1a_2c_0d_2+a_2b_1d_0d_2  \cr s''_0=1,\quad s''_1=a_1,\quad s''_2=b_1b_2c_2+a_1a_2d_2 \end{cases}$$
and
$$\begin{cases} s_n=a_nd_ns_{n-1}+b_nb_{n-1}c_nd_{n-1}s_{n-2}+b_nb_{n-1}b_{n-2}d_nd_{n-1}d_{n-2}s_{n-3} \cr s'_n=a_nd_ns'_{n-1}+b_nb_{n-1}c_nd_{n-1}s'_{n-2}+b_nb_{n-1}b_{n-2}d_nd_{n-1}d_{n-2}s'_{n-3} \cr s''_n=a_nd_ns''_{n-1}+b_nb_{n-1}c_nd_{n-1}s''_{n-2}+b_nb_{n-1}b_{n-2}d_nd_{n-1}d_{n-2}s''_{n-3} \end{cases},\quad \forall n\geq 3.$$
For these results see \cite{abcm}. Thus, a ternary continued fraction with rational partial quotients can be represented by matrices with integer entries like
$$\begin{pmatrix} a_id_i & b_id_i & 0 \cr c_ib_i & 0 & b_id_i \cr b_id_i & 0 & 0 \end{pmatrix},$$
which play the same role of the matrices used in \eqref{matrix-fcb}.
\end{remark}
The periodic expansion of Theorem \ref{fcb-teo1} has been found starting from the development of
\begin{equation} \label{nu} (z+\sqrt[3]{d^2})^n=\nu_n^{(0)}+\nu_n^{(1)}\sqrt[3]{d}+\nu^{(2)}_n\sqrt[3]{d^2}, \end{equation}
for every integer $z\not=0$, $d$ integer not cube, and where $\nu_n^{(0)}, \nu_n^{(1)}, \nu^{(2)}_n$ are polynomials such that 
$$\lim_{n\rightarrow}\cfrac{\nu_n^{(0)}}{\nu_n^{(2)}}=\sqrt[3]{d^2},\quad \lim_{n\rightarrow}\cfrac{\nu_n^{(1)}}{\nu_n^{(2)}}=\sqrt[3]{d}$$
These ratios generalize the R\'edei rational functions \cite{Redei}. Indeed R\'edei rational functions arise from the development of
$$(z+\sqrt{d})^n=N_n(d,z)+D_n(d,z)\sqrt{d},$$
for every integer $z\not=0$, $d$ integer not square, and where
$$ N_n(d,z)=\sum_{k=0}^{[n/2]}\binom{n}{2k}d^kz^{n-2k},\quad D_n(d,z)=\sum_{k=0}^{[n/2]}\binom{n}{2k+1}d^kz^{n-2k-1}. $$
The R\'edei rational functions are defined as
$$Q_n(d,z)=\cfrac{N_n(d,z)}{D_n(d,z)}, \quad \forall n\geq 1.$$
The R\'edei rational functions are very interesting and useful tools in number theory. Indeed they are permutation functions of finite fields (see, e.g., \cite{Lidl}) and they can be also used in order to generate pseudorandom sequences \cite{pseudo} or to construct a public key cryptographic system \cite{crit}. Moreover in \cite{bcm} and \cite{abcm2}, R\'edei rational functions are connected to periodic continued fractions with rational partial quotients convergent to square roots. In \cite{abcm}, R\'edei rational functions have been generalized in order to obtain periodic representations only for cubic roots. In the next section, we will propose a different generalization of the R\'edei rational functions in order to construct periodic ternary continued fractions convergent to any cubic irrationalities. 

\section{The main case}
Let $\alpha$ be a real root of the polynomial $x^3-px^2-qx-r$, with $p,q,r\in\mathbb Q$. Let us consider
\begin{equation} \label{cerruti}(z+\alpha^2)^n=\mu_n^{(0)}+\mu_n^{(1)}\alpha+\mu_n^{(2)}\alpha^2, \end{equation}
for $z$ integer number not zero and where the polynomials $\mu_n^{(i)}$ depends on $p,q,r,z$ and we will call it \emph{Cerruti polynomials}, since when $\alpha=\sqrt[3]{d}$ they are the polynomials $\nu_n^{(i)}$ \eqref{nu} introduced the first time in \cite{abcm}. Let $N$ be the following fundamental matrix
\begin{equation} \label{mn} N=\begin{pmatrix} z & r & pr \cr 0 & q+z & pq+r \cr 1 & p & p^2+q+z \end{pmatrix}. \end{equation}
Its characteristic polynomial is
$$x^3-\text{Tr}(N)x^2+\cfrac{1}{2}(\text{Tr}(N)^2-\text{Tr}(N^2))x-\det (N),$$
i.e.,
$$x^3-(p^2+2q+3z)x^2+(q^2-2pr+2p^2z+4qz+3z^2)x-(r^2+q^2z-2prz+p^2z^2+2qz^2+z^3).$$
In the following, we set $I_1(N)=\cfrac{1}{2}(\text{Tr}(N)^2-\text{Tr}(N^2))$.
\begin{theorem} \label{zN}
Let $N$ and $\mu_n^{(i)}$ be the matrix \eqref{mn} and the Cerruti polynomials above defined.
\begin{enumerate}
\item The characteristic polynomial of $N$ has roots
$$z+\alpha_1^2,\quad z+\alpha_2^2,\quad z+\alpha_3^2,$$
where $\alpha_1,\alpha_2,\alpha_3$ are the roots of $x^3-px^2-qx-r$.
\item $$N^n=\begin{pmatrix}  \mu_n^{(0)} & r\mu_n^{(2)} & r\mu_n^{(1)}+pr\mu_n^{(2)} \cr \mu_n^{(1)} & \mu_n^{(0)}+q\mu_n^{(2)} & (pq+r)\mu_n^{(2)}+q\mu_n^{(1)} \cr \mu_n^{(2)} & \mu_n^{(1)}+p\mu_n^{(2)} & \mu_n^{(0)}+p\mu_n^{(1)}+(p^2+q)\mu_n^{(2)} \end{pmatrix}$$
\end{enumerate}
\end{theorem}
\begin{proof}
\begin{enumerate}
\item Considering that
$$\alpha_1\alpha_2\alpha_3=r,\quad \alpha_1+\alpha_2+\alpha_3=p,\quad \alpha_1\alpha_2+\alpha_2\alpha_3+\alpha_1\alpha_3=-q,$$
we have
$$\alpha_1^2+\alpha_2^2+\alpha_3^2=p^2+2q,\quad \alpha_1^2\alpha_2^2+\alpha_2^2\alpha_3^2+\alpha_1^2\alpha_3^2=q^2-2pr.$$
Moreover, expanding the polynomial $(x-(z+\alpha_1^2))(x-(z+\alpha_2^2))(x-(z+\alpha_3^2))$, it is easy to see that the coefficient of $x^2$, the coefficient of $x$, and the constant term are
$$-(\alpha_1^2+\alpha_2^2+\alpha_3^2+3z)=-\text{Tr}(N)$$
$$\alpha_1^2\alpha_2^2+\alpha_2^2\alpha_3^2+\alpha_1^2\alpha_3^2+2z(\alpha_1^2+\alpha_2^2+\alpha_3^2)+3z^2=I_1(N)$$
$$-\alpha_1^2\alpha_2^2\alpha_3^2-z(\alpha_1^2\alpha_2^2+\alpha_2^2\alpha_3^2+\alpha_1^2\alpha_3^2)-z^2(\alpha_1^2+\alpha_2^2+\alpha_3^2)-z^3=-\det (N),$$
respectively.
\item By definition of Cerruti polynomials \eqref{cerruti}, it follows that $\mu_n^{(i)}$'s, for $i=0,1,2$, are linear recurrent sequences of degree 3 whose characteristic polynomial is the minimal polynomial of $z+\alpha^2$ (where $\alpha$ real root of $x^3-px^2-qx-r$), i.e., the characteristic polynomial of $N$.\\
Thus, we only have to check the initial conditions. We start from
$$(z+\alpha^2)^0=1,$$
i.e.
$$\mu_0^{(0)}=1,\quad \mu_0^{(1)}=0,\quad \mu_0^{(2)}=0$$
and since
$$N^0=\begin{pmatrix} 1 & 0 & 0 \cr 0 & 1 & 0 \cr 0 & 0 & 1 \end{pmatrix}$$
the initial condition for $n=0$ is satisfied. Considering $z+\alpha^2$, it follows that
$$\mu_1^{(0)}=z,\quad \mu_1^{(1)}=0,\quad \mu_1^{(2)}=1.$$
Thus,
$$N=\begin{pmatrix} z & r & pr \cr 0 & q+z & pq+r \cr 1 & p & p^2+q+z \end{pmatrix}=\begin{pmatrix}  \mu_1^{(0)} & r\mu_1^{(2)} & r\mu_1^{(1)}+pr\mu_1^{(2)} \cr \mu_1^{(1)} & \mu_1^{(0)}+q\mu_1^{(2)} & (pq+r)\mu_1^{(2)}+q\mu_1^{(1)} \cr \mu_1^{(2)} & \mu_1^{(1)}+p\mu_1^{(2)} & \mu_1^{(0)}+p\mu_1^{(1)}+(p^2+q)\mu_1^{(2)} \end{pmatrix}.$$
Finally,
$$(z+\alpha^2)^2=z^2+2\alpha^2z+\alpha^4=z^2+pr+(pq+r)\alpha+(p^2+q+2z)\alpha^2$$
and
$$N^2=\begin{pmatrix} pr+z^2 & p^2r+qr+2rz & p^3r+2pqr+r^2+2prz \cr pq+r & p^2q+q^2+pr+2qz+z^2 & p^3q+2pq^2+p^2r+2qr+2pqz+2rz \cr p^2+q+2z & p^3+2pq+r+2pz & p^4+3p^2q+q^2+2pr+2p^2z+2qz+z^2  \end{pmatrix}.$$
\end{enumerate}
\end{proof}
\begin{theorem} \label{binet}
Let $\alpha$ be a real root largest in modulus of $x^3-px^2-qx-r$ and let $\alpha_2,\alpha_3$ be the remaining roots. Let $\mu_n^{(i)}$ be the Cerruti polynomials \eqref{cerruti}, then is
$$\lim_{n\rightarrow\infty}\cfrac{\mu_n^{(0)}}{\mu_n^{(2)}}=\cfrac{r}{\alpha},\quad \lim_{n\rightarrow\infty}\cfrac{\mu_n^{(1)}}{\mu_n^{(2)}}=\alpha-p,$$
for any integer $z$ such that $z+\alpha^2$ larger in modulus than $z+\alpha_2^2, z+\alpha_3^2$ and $\mu_n^{(2)}\not=0$.
\end{theorem} 
\begin{proof}
Let  
$$\beta_1=z+\alpha^2,\quad \beta_2=z+\alpha_2^2,\quad \beta_3=z+\alpha_3^2$$
be the roots of the characteristic polynomial of $N$. By the Binet formula
$$\begin{cases} \mu_n^{(0)}=a_1\beta_1^n+a_2\beta_2^n+a_3\beta_3^n \cr \mu_n^{(1)}=b_1\beta_1^n+b_2\beta_2^n+b_3\beta_3^n \cr \mu_n^{(2)}=c_1\beta_1^n+c_2\beta_2^n+c_3\beta_3^n \end{cases},\quad \forall n\geq 0$$
where the coefficients $a_i,b_i,c_i$ can be obtained by initial conditions, solving the system
$$\begin{cases} a_1+a_2+a_3=1 \cr a_1\beta_1+a_2\beta_2+a_3\beta_3=z,\cr a_1\beta_1^2+a_2\beta_2^2+a_3\beta_3^2=pr+z^2 \end{cases}$$
and similar systems for the $b_i$'s and $c_i$'s.
Since $\beta_1$ is larger in modulus than $\beta_2,\beta_3$, we are only interested in
$$\begin{cases} a_1= \cfrac{\beta_2\beta_3-z(\beta_2+\beta_3)+pr+z^2}{(\beta_1-\beta_2)(\beta_1-\beta_3)} \cr b_1=\cfrac{pq+r}{(\beta_1-\beta_2)(\beta_1-\beta_3)} \cr c_1=\cfrac{2z+p^2+q-(\beta_2+\beta_3)}{(\beta_1-\beta_2)(\beta_1-\beta_3)} \end{cases}.$$
Now,
$$\lim_{n\rightarrow\infty}\cfrac{\mu_n^{(1)}}{\mu_n^{(2)}}=\cfrac{b_1}{c_1}=\cfrac{pq+r}{2z+p^2+q-(2z+\alpha_2^2+\alpha_3^2)}=\cfrac{pq+r}{\alpha^2-q},$$
moreover
$$(\alpha^2-q)(\alpha-p)=\alpha^3-p\alpha^2-q\alpha+pq=pq+r$$
and it is proved that 
$$\lim_{n\rightarrow\infty}\cfrac{\mu_n^{(1)}}{\mu_n^{(2)}}=\alpha-p.$$
The proof that $\lim_{n\rightarrow\infty}\cfrac{\mu_n^{(0)}}{\mu_n^{(2)}}=\cfrac{r}{\alpha}$ is left to the reader.
\end{proof}
\begin{remark}
It is always possible to find integers $z$ satisfying the condition of Theorem \ref{binet} and such that $\mu_n^{(2)}\not=0$.
\end{remark}
\begin{theorem} \label{main}
Let $\alpha$ be a real root largest in modulus of $x^3-px^2-qx-r$ and $N$ the matrix defined in \eqref{mn}, then
\begin{multline}\label{hermite} [\{z,\cfrac{2z+p^2+q}{pq+r},\overline{\cfrac{(pq+r)\text{Tr}(N)}{\det (N)},\text{Tr}(N),\cfrac{\text{Tr}(N)}{pq+r}}\}, \\ \{p,-\cfrac{z^2+qz+p^2z-pr}{pq+r},\overline{-\cfrac{I_1(N)}{\det (N)},-\cfrac{(pq+r)I_1(N)}{\det(N)},-\cfrac{I_1(N)}{pq+r}}\}] = 
(\cfrac{r}{\alpha},\alpha)  \end{multline}
for any integer $z$ satisfying the hypothesis of Theorem \ref{binet}.
\end{theorem}
\begin{proof}
By Theorem \ref{binet}, it is sufficient to prove that 
$$\cfrac{A_n}{C_n}=\cfrac{\mu_{n+1}^{(0)}}{\mu_{n+1}^{(2)}},\quad \cfrac{B_n}{C_n}=\cfrac{\mu_{n+1}^{(1)}}{\mu_{n+1}^{(2)}}+p,\quad \forall n\geq0,$$
where $\cfrac{A_n}{C_n},\cfrac{B_n}{C_n}$ are the convergents of the ternary continued fraction \eqref{hermite} satisfying Eqs. \eqref{convergenti}. First of all we prove by induction that 
$$A_n=\cfrac{\mu_{n+1}^{(0)}}{(pq+r)^k\det(N)^{[\frac{n+1}{3}]}},\quad \forall n\geq0,$$
where
$$k=\begin{cases} 1 \quad n\equiv 1\pmod 3 \cr 0\quad \text{otherwise}  \end{cases}$$
The inductive basis is straightforward to prove, indeed
$$A_0=z, \quad A_1=\cfrac{pr+z^2}{pq+r},\quad A_2=\cfrac{p^3r+2pqr+r^2+3prz+z^3}{r^2+q^2z-2prz+p^2z^2+2qz^2+z^3}.$$
Now, for $n\geq3$, we consider the cases
$$n\equiv0 \pmod 3,\quad n\equiv 1 \pmod 3,\quad n\equiv 2 \pmod 3.$$
Let us consider $n\equiv0 \pmod 3$, then
$$A_n=\text{Tr}(N)A_{n-1}-\cfrac{(pq+r)I_1(N)}{\det(N)}A_{n-2}+A_{n-3},$$
by inductive hypothesis we have
$$A_n=\text{Tr}\cfrac{\mu_n^{(0)}}{\det(N)^{[\frac{n}{3}]}}-\cfrac{(pq+r)I_1(N)}{\det(N)}\cdot\cfrac{\mu_{n-1}^{(0)}}{(pq+r)\det(N)^{[\frac{n-1}{3}]}}+\cfrac{\mu_{n-2}^{(0)}}{\det(N)^{[\frac{n-2}{3}]}}.$$
Since $n\equiv0 \pmod 3$, we have
$$[\cfrac{n}{3}]=[\cfrac{n+1}{3}],\quad [\cfrac{n-1}{3}]=[\cfrac{n-2}{3}]=[\cfrac{n+1}{3}]-1.$$
Using the recurrence relation for the Cerruti polynomials, we obtain
$$A_n=\cfrac{\text{Tr}(N)\mu_n^{(0)}-I_1(N)\mu_{n-1}^{(0)}+\det(N)\mu_{n-2}^{(0)}}{\det(N)^{[\frac{n+1}{3}]}}=\cfrac{\mu_{n+1}^{(0)}}{(pq+r)^k\det(N)^{[\frac{n+1}{3}]}}.$$
Let us consider $n\equiv1 \pmod 3$, then
$$A_n=\cfrac{\text{Tr(N)}}{pq+r}A_{n-1}-\cfrac{I_1(N)}{pq+r}A_{n-2}+A_{n-3}$$
and
$$[\frac{n+1}{3}]=[\cfrac{n}{3}]=[\frac{n-1}{3}]=[\frac{n-2}{3}]+1.$$
Thus, we easily obtain
$$A_n=\cfrac{\text{Tr}(N)\mu_{n}^{(0)}-I_1(N)\mu_{n-1}^{(0)}+\det(N)\mu_{n-2}^{(0)}}{(pq+r)\det(N)^{[\frac{n+1}{3}]}}=\cfrac{\mu_{n+1}^{(0)}}{(pq+r)^k\det(N)^{[\frac{n+1}{3}]}}.$$
Similarly when $n\equiv2 \pmod 3$.
In a similar way, it is possible to prove that
$$C_n=\cfrac{\mu_{n+1}^{(2)}}{(pq+r)^k\det(N)^{[\frac{n+1}{3}]}},\quad \forall n\geq0.$$
Consequently,
$$\lim_{n\rightarrow\infty}\cfrac{A_n}{C_n}=\lim_{n\rightarrow\infty}\cfrac{\mu_{n+1}^{(0)}}{\mu_{n+1}^{(2)}}=\cfrac{r}{\alpha}.$$
Finally, let us consider the sequence $(B_n)_{n=0}^\infty$. We prove by induction that
$$B_n=\cfrac{\mu_{n+1}^{(1)}+p\mu_{n+1}^{(2)}}{(pq+r)^k\det(N)^{[\frac{n+1}{3}]}},\quad \forall n\geq0.$$
The steps $n=0,1,2$ can be directly checked. Let us consider $n\geq3$ and $n\equiv 2\pmod 3$ (similarly the formula can be proved when $n\equiv 0\pmod 3$ and $n\equiv 1\pmod 3$). We have
$$B_n=\cfrac{(pq+r)\text{Tr}(N)}{\det(N)}B_{n-1}-\cfrac{I_1(N)}{\det(N)}B_{n-2}+B_{n-3}=$$
$$=\cfrac{(pq+r)\text{Tr}(N)}{\det(N)}\cdot\cfrac{\mu_n^{(1)}+p\mu_n^{(2)}}{(pq+r)\det(N)^{\frac{n}{3}}}-\cfrac{I_1(N)}{\det(N)}\cdot\cfrac{\mu_{n-1}^{(1)}+p\mu_{n-1}^{(2)}}{(pq+r)\det(N)^{\frac{n-1}{3}}}+\cfrac{\mu_{n-2}^{(1)}+p\mu_{n-2}^{(2)}}{(pq+r)\det(N)^{\frac{n}{3}}}.$$
Since
$$[\cfrac{n+1}{3}]=[\cfrac{n}{3}]+1=[\cfrac{n-1}{3}]+1=[\cfrac{n-2}{3}]+1$$
the formula is proved. Thus,
$$\lim_{n\rightarrow\infty}\cfrac{B_n}{C_n}=\lim_{n\rightarrow\infty}\cfrac{\mu_{n+1}^{(1)}}{\mu_{n+1}^{(2)}}+p=\alpha.$$
\end{proof}
The previous theorem provides a periodic representation for all the cubic irrationalities $\alpha$, such that $\alpha$ is the root greatest in modulus of its minimal polynomial. This representation is a ternary continued fraction \eqref{fcb} of period 3 whose partial quotients are given by \eqref{hermite}. Clearly, this expansion leads to periodic sequences of integer numbers that represent cubic irrationals. This fact is highlighted by using the matricial approach. In particular, the irrationalities $(\cfrac{r}{\alpha},\alpha)$ are represented by a periodic product of matrices \eqref{matrix-fcb} whose entries are rational numbers, or equivalently by a periodic product of matrices \eqref{rm} whose entries are integer numbers. 
\begin{example}
Let us consider the cubic polynomial $x^3-5x^2+x-3$, having a real root
$$\alpha=\cfrac{1}{3}(5+\sqrt[3]{44}+\sqrt[3]{242})\simeq 4.9207,$$
greater in modulus than the complex roots. Theorem \ref{main} provides periodic representations of $\alpha$. If we choose, e.g., $z=5$ (which satisfies conditions of Theorem \ref{binet}), we have
$$N=\begin{pmatrix} 5 & 3 & 15 \cr 0 & 4 & -2 \cr 1 & 5 & 29 \end{pmatrix}$$
and
$$(\cfrac{3}{\alpha},\alpha)=[\{5,-17,\overline{-\frac{19}{141},38,-19}\},\{5,65,\overline{-\frac{23}{47},-\frac{46}{47},138}\}].$$
Moreover, we can use \eqref{rm} in order to express the cubic irrationality $\alpha$ as a periodic product of matrices with integer entries 
$$\begin{pmatrix} 5 & 1 & 0 \cr 5 & 0 & 1 \cr 1 & 0 & 0 \end{pmatrix}\begin{pmatrix} -17 & 1 & 0 \cr 65 & 0 & 1 \cr 1 & 0 & 0 \end{pmatrix}\overline{\begin{pmatrix} -893 & 6627 & 0 \cr -3243 & 0 & 6627 \cr 6627 & 0 & 0 \end{pmatrix}\begin{pmatrix} 1786 & 47 & 0 \cr 46 & 0 & 47 \cr 47 & 0 & 0 \end{pmatrix}\begin{pmatrix} -19 & 1 & 0 \cr 138 & 0 & 1 \cr 1 & 0 & 0 \end{pmatrix}}.$$
This periodic representation provides rational approximations for $\alpha$. For example, the convergents of the ternary continued fraction can be evaluated by using \eqref{convergenti} and we obtain rational approximations of $\cfrac{3}{\alpha}$ and $\alpha$, respectively:
$$(5,\cfrac{20}{17},\cfrac{88}{127},\cfrac{4633}{7447},\cfrac{66559}{108838},...)=(5, 1.1764, 0.6929, 0.6221, 0.6115,...)$$
$$(5,\cfrac{84}{17},\cfrac{1251}{254},\cfrac{36651}{7447},\cfrac{535575}{108838})=(5, 4.9412, 4.9252, 4.9216, 4.9208,...)$$
These rational approximations can be obviously obtained by the matricial representation. For example if we set $A$ for the matrix of the pre--period (i.e., the matrix product of the two matrices of the pre--period) and $P$ for the matrix of the period (i.e., the matrix product of the three matrices of the period), then 
$$AP=\begin{pmatrix} -147028831 & 10234297 & 388784 \cr -1183085175 & 80962059 & 2763459 \cr -240423142 & 16450423 & 561086 \end{pmatrix}$$
and
$$\cfrac{-147028831}{-240423142}=\cfrac{66559}{108838},\quad \cfrac{-1183085175}{-240423142}=\cfrac{535575}{108838}.$$ 
\end{example}
\begin{example}
Let us consider the cubic polynomial $3x^3-12x^2-4x+1$. We can apply Theorem \ref{main} with $p=4, q=4/3, r=-1/3$. Using $z=1$, we obtain
$$(-\cfrac{1}{3\alpha},\alpha)=[\{1,\cfrac{58}{15},\overline{\cfrac{975}{218},\cfrac{65}{3},\cfrac{13}{3}}\},\{4,-\cfrac{59}{15},\overline{-\cfrac{403}{218},-\cfrac{2015}{218},-\cfrac{403}{45}}\}].$$
that is a periodic representation of $\alpha$ root greatest in modulus of $3x^3-12x^2-4x+1$.\\
If we choose, e.g., $z=-1$, we obtain a different periodic ternary continued fraction convergent to $\alpha$:
$$(-\cfrac{1}{3\alpha},\alpha)=[\{-1,\cfrac{46}{15},\overline{\cfrac{47}{8},\cfrac{47}{3},\cfrac{47}{15}}\},\{4,3,\overline{-\cfrac{269}{120},\cfrac{269}{24},\cfrac{269}{45}}\}],$$
where $\alpha\simeq 4.29253$. These periodic representations leads to a periodic product of matrices with integer entries convergent to the irrational $\alpha$.
\end{example}

\section{The remaining cases}
When $\alpha$ is the root largest in modulus of a general cubic polynomial $x^3-px^2-qx-r$, a periodic representation for $\alpha$ is provided by Theorem \ref{main}. In this section, we treat all the remaining cases of cubic irrationalities.\\
\indent If $\alpha$ is the root smallest in modulus of $x^3-px^2-qx-r$, Theorem \ref{main} does not work. In this case, we consider the reflected polynomial $x^3+\cfrac{q}{r}x^2+\cfrac{p}{r}x-\cfrac{1}{r}$ whose roots are $\cfrac{1}{\alpha}, \cfrac{1}{\alpha_2}, \cfrac{1}{\alpha_3}$. In this way $\cfrac{1}{\alpha}$ is the root greatest in modulus of $x^3+\cfrac{q}{r}x^2+\cfrac{p}{r}x-\cfrac{1}{r}$ and by Theorem \ref{main} we get a periodic representation for the couple $(\cfrac{\alpha}{r},\cfrac{1}{\alpha})$. Successively, a periodic representation for $\alpha$ can be derived. We need the following
\begin{theorem} \label{r}
Let $[\{a_0,a_1,\overline{a_2,a_3,a_4}\},\{b_0,b_1,\overline{b_2,b_3,b_4}\}]$ be a periodic ternary continued fraction that converges to a couple of real number $(\alpha,\beta)$, then the periodic ternary continued fraction
\begin{equation}  \label{rfcb} [\{ra_0,ra_1,\overline{\frac{a_2}{r^2},ra_3,ra_4}\},\{\frac{b_0}{r},r^2b_1,\overline{\frac{b_2}{r},\frac{b_3}{r},r^2b_4}\}] \end{equation}
converges to the couple of real number $(r\alpha,\cfrac{1}{r\beta})$, for $r$ rational number.
\end{theorem}
\begin{proof}
Let $\frac{A_n}{C_n}$ and $\frac{B_n}{C_n}$ be the $n$--th convergents of the ternary continued fraction of $(\alpha,\beta)$. Let $\frac{\tilde A_n}{\tilde C_n}$ and $\frac{\tilde B_n}{\tilde C_n}$ be the $n$--th convergents of the ternary continued fractions \eqref{rfcb}, then 
$$\tilde A_n=r^{k_1}A_n,\quad \tilde B_n=r^{k_2}B_n,\quad \tilde C_n=r^{k_3}C_n,\quad \forall n\geq0$$
where
$$k_1=\begin{cases} 0,\quad n\equiv 2 \pmod 3 \cr 1,\quad n\equiv 0 \pmod 3 \cr 2,\quad n\equiv 1 \pmod 3 \end{cases}, \quad k_2=\begin{cases} -1,\quad n\equiv 2 \pmod 3 \cr 0,\quad n\equiv 0 \pmod 3 \cr 1,\quad n\equiv 1 \pmod 3 \end{cases}$$
$$k_3=\begin{cases} -2,\quad n\equiv 2 \pmod 3 \cr -1,\quad n\equiv 0 \pmod 3 \cr 0,\quad n\equiv 1 \pmod 3 \end{cases}$$
It is straightforward to check these identities for $n=0,1,2$. Let us proceed by induction, considering an integer $m\equiv 0\pmod 3$. Then
$$\tilde A_m=ra_3\tilde A_{m-1}+\frac{b_3}{r}\tilde A_{m-2}+\tilde A_{m-3}=ra_3A_{m-1}+b_3rA_{m-2}+rA_{m-3}=rA_m.$$
Similarly when $m\equiv 1\pmod 3$ and $m\equiv 2\pmod 3$, and for the sequences $\tilde B_n$ and $\tilde C_n$.
\end{proof}
Thus, if $\alpha$ is the root smallest in modulus of $x^3-px^2-qx-r$, by Theorem \ref{main} we get the periodic ternary continued fraction of $(\cfrac{\alpha}{r},\cfrac{1}{\alpha})$ and by Theorem \ref{r} we get the periodic ternary continued fraction of $(\alpha,\cfrac{1}{r\alpha})$.
\begin{example}
Let us consider the cubic polynomial $x^3-2x^2+x+1$. It has one real root $\alpha$ whose modulus is smaller than the modulus of the complex roots. Theorem \ref{main} does not work on this polynomial, but we can consider the reflected polynomial $x^3+x^2-2x+1$ whose roots are the inverse roots of $x^3-2x^2+x+1$. Thus, $\cfrac{1}{\alpha}$ is the real root of $x^3+x^2-2x+1$ largest in modulus and we can apply Theorem \ref{main}. Posing, e.g., $z=5$, we obtain
$$(-\alpha,\cfrac{1}{\alpha})=[\{5,-\cfrac{13}{3},\overline{-\cfrac{20}{87},20,-\cfrac{20}{3}}\},\{-1,13,\overline{-\cfrac{127}{261},\cfrac{127}{87},\cfrac{127}{3}}\}].$$
Finally, we multiply by -1 this ternary continued fraction and by Theorem \ref{r} we obtain
$$(\alpha,-\cfrac{1}{\alpha})=[\{-5,\cfrac{13}{3},\overline{-\cfrac{20}{87},-20,\cfrac{20}{3}}\},\{1,13,\overline{\cfrac{127}{261},-\cfrac{127}{87},\cfrac{127}{3}}\}],$$
i.e., we found a periodic representation for $\alpha$ root smallest in modulus of $x^3-2x^2+x+1$.
\end{example}
Now, we are able to determine a periodic representation for any cubic irrational that is the root largest or smallest in modulus of a cubic polynomial $x^3-px^2-qx-r$. \\
\indent Finally, we treat the last case, i.e., $\alpha$ is the intermediate root of a cubic polynomial having three real roots. Let $\alpha_1,\alpha_2,\alpha_3$ be the real root of $x^3-px^2-qx-r$ such that $|\alpha_3|<|\alpha_2|<|\alpha_1|$. Using previous techniques we can get periodic expansions for $\alpha_1$ and $\alpha_3$. Moreover, a rational number $k$ can be ever found such that $\alpha_2\pm k$ is the root largest or smallest in modulus of its minimal polynomial $(x-(\alpha_1\pm k))(x-(\alpha_2\pm k))(x-(\alpha_3\pm k))$. The coefficients of $(x-(\alpha_1\pm k))(x-(\alpha_2\pm k))(x-(\alpha_3\pm k))$ can be derived from the coefficients of $x^3-px^2-qx-r$ (see, e.g., Th. 10 and Cor. 11 \cite{bcm2}). Thus, by Theorem \ref{main} we know the periodic ternary continued fraction of $(\cfrac{r'}{\alpha_2\pm k},\alpha_2\pm k)$. Then it is immediate to obtain the periodic expansion of $(\cfrac{r'}{\alpha_2\pm k},\alpha_2)$ (see \eqref{fcb}).
\begin{example}
Let us consider the Ramanujan cubic polynomial $x^3+x^2-2x-1$ with three real roots. The roots of this polynomial are quite famous (see, e.g., \cite{WS}) and it is well--known that they are
$$\alpha_1=2\cos \cfrac{2\pi}{7},\quad \alpha_2=2\cos \cfrac{4\pi}{7},\quad \alpha_3=2\cos \cfrac{8\pi}{7},$$
where
$$\alpha_1 \simeq 1.24698,\quad \alpha_2\simeq -0.445042,\quad \alpha_3\simeq -1.80194.$$
Thus, from Theorem \ref{main}, for $z=3$, we obtain
$$(\cfrac{1}{\alpha_3},\alpha_3)=[\{3,-9,\overline{-\cfrac{2}{13},14,-14}\},\{-1,19,\overline{-\cfrac{9}{13},\cfrac{9}{13},63}\}].$$
Considering the polynomial $x^3+2x^2-x-1$ and $z=1$, we obtain
$$(\alpha_2,\cfrac{1}{\alpha_2})=[\{1,-7,\overline{-\cfrac{9}{13},9,-9}\},\{-2,8,\overline{-\cfrac{20}{13},\cfrac{20}{13},20}\}].$$
Finally, we can consider the minimal polynomial of 
$$\alpha_1+1,\quad \alpha_2+1,\quad \alpha_3+1$$
that is $x^3-2x^2-x+1$, whose root largest in modulus is $\alpha_1+1$. For $z=2$, by Theorem \ref{main}, we obtain
$$(-\cfrac{1}{\alpha_1+1},\alpha_1+1)=[\{2,9,\overline{\cfrac{12}{43},12,12}\},\{2,-16,\overline{-\cfrac{41}{43},-\cfrac{41}{43},-41}\}]$$
and
$$(-\cfrac{1}{\alpha_1+1},\alpha_1)=[\{2,9,\overline{\cfrac{12}{43},12,12}\},\{1,-16,\overline{-\cfrac{41}{43},-\cfrac{41}{43},-41}\}].$$
\end{example}

\section{The periodic algorithm}
An approach to the Hermite problem contemplates the research of a function whose iteration on algebraic irrationalities provides a periodical algorithm. The partial quotients of the ternary continued fraction \eqref{hermite} can be derived from the Jacobi algorithm \eqref{algobiforcanti} using two functions $f_z^\alpha,g_z^\alpha$ instead of the floor function.
\begin{definition} \label{map}
Let $\alpha$ and $\mathbb Q(\alpha)$ be a root of $x^3-px^2-qx-r$ and the algebraic extension of $\mathbb Q$, respectively. We define the linear functions $f_z^\alpha,g_z^\alpha:\mathbb Q (\alpha)\rightarrow \mathbb Q$, for $z\in\mathbb Z$, such that
\begin{enumerate}
\item $f_z^\alpha(q)=g_z^\alpha(q)=q,\quad \forall q\in\mathbb Q$
\item $f_z^\alpha\left(\cfrac{r}{\alpha}\right)=g_z^\alpha\left(\cfrac{r}{\alpha}\right)=z$
\item $f_z^\alpha(\alpha)=g_z^\alpha(\alpha)=p$
\item $f_z^\alpha(\alpha^2)=2z+p^2+2q, \quad g_z^\alpha(\alpha^2)=z+p^2+q$
\end{enumerate}
\end{definition}
The ternary continued fraction \eqref{hermite} is obtained from the following algorithm:
\begin{equation} \label{fzgz} \begin{cases} a_n=f_z^\alpha(x_n) \cr b_n=g_z^\alpha(y_n) \cr x_{n+1}=\cfrac{1}{y_n-b_n} \cr y_{n+1}=\cfrac{x_n-a_n}{y_n-b_n}  \end{cases}, \end{equation}
for $n=0,1,2,...$ and $x_0=\cfrac{r}{\alpha}$, $y_0=\alpha$, where $\alpha$ root of the polynomial $x^3-px^2-qx-r$.\\
Let us start to use the algorithm \eqref{fzgz} with inputs $(\cfrac{r}{\alpha},\alpha)$, we immediately have
$$a_0=z,\quad b_0=p.$$
Now, we evaluate $x_1$ and $y_1$:
$$x_1=\cfrac{1}{\alpha-p},\quad y_1=\cfrac{\cfrac{r}{\alpha}-z}{\alpha-p}.$$
We need to manipulate $x_1$ and $y_1$ in order to find the values of $f_z^\alpha(x_1)$ and $g_z^\alpha(y_1)$. In particular, we will often use that $\alpha^3=p\alpha^2+q\alpha+r$. We have
$$x_1=\cfrac{1}{\alpha-p}\cdot\cfrac{\alpha^2-q}{\alpha^2-q}=\cfrac{\alpha^2-q}{pq+r}$$
$$y_1=\cfrac{\cfrac{r}{\alpha}-z}{\alpha-p}\cdot\cfrac{\alpha^2-q}{\alpha^2-q}=\cfrac{r\alpha-z\alpha^2+qz-\cfrac{qr}{\alpha}}{pq+r}.$$
Now we can apply the properties of $f_z^\alpha$ and $g_z^\alpha$ and we obtain
$$a_1=f_z^\alpha(x_1)=\cfrac{2z+p^2+q}{pq+r},\quad b_1=g_z^\alpha(y_1)=-\cfrac{z^2+qz+p^2z-pr}{pq+r}.$$
Let us continue with
$$x_2=\cfrac{1}{y_1-b_1}=\cfrac{pq+r}{(\cfrac{r}{\alpha}-z)(\alpha^2-q)+(z^2+qz+p^2z-pr)(\alpha^2-q)}\cdot\cfrac{\alpha^2+z}{\alpha^2+z}=\cfrac{(pq+r)(\alpha^2+z)}{\det(N)},$$
where the last identities follow by using $\alpha^3=p\alpha^2+q\alpha+r$. Moreover,
$$y_2=(x_1-a_1)\cdot x_2=\cfrac{(\alpha^2-q)(\alpha^2+z)-(2z+p^2+q)(\alpha^2+z)}{\det(N)}$$
and
$$a_2=\cfrac{(pq+r)(3z+p^2+2q)}{\det(N)}=\cfrac{(pq+r)\text{Tr}(N)}{\det(N)},\quad b_2=-\cfrac{I_1(N)}{\det(N)}.$$
Then,
$$x_3=\cfrac{\det(N)}{(\alpha^2-q)(\alpha^2+z)-(2z+p^2+q)(\alpha^2+z)+I_1(N)}\cdot\cfrac{\alpha^2+z}{\alpha^2+z}=\cfrac{\det(N)(\alpha^2+z)}{\det(N)}=\alpha^2+z,$$
$$y_3=\cfrac{(pq+r)((\alpha^2+z^2)^2-\text{Tr}(N)(\alpha^2+z))}{\det(N)}$$
and
$$a_3=3z+2q+p^2=\text{Tr}(N),\quad b_3=-\cfrac{(pq+r)I_1(N)}{\det(N)}.$$
Finally,
$$x_4=\cfrac{\det(N)}{(pq+r)((\alpha^2+z)^2-\text{Tr}(N)(\alpha^2+z)+I_1(N))}\cdot\cfrac{\alpha^2+z}{\alpha^2+z}=\cfrac{\alpha^2+z}{pq+r},$$
where the last identity is obtained recalling that $\alpha^2+z$ is the root of the characteristic polynomial of $N$. 
$$y_4=(\alpha^2+z-\text{Tr(N)})\cdot\cfrac{\alpha^2+z}{pq+r}$$
from which
$$a_4=\cfrac{\text{Tr}(N)}{pq+r},\quad b_4=-\cfrac{I_1(N)}{pq+r}.$$
Now we check that $x_5=x_2$ and $y_5=y_2$:
$$x_5=\cfrac{pq+r}{(\alpha^2+z)^2-\text{Tr}(N)(\alpha^2+z)+I_1(N)}=\cfrac{(pq+r)(\alpha^2+z)}{\det(N)}=x_2$$
$$y_5=\cfrac{\alpha^2+z-\text{Tr}(N)}{pq+r}\cdot\cfrac{(pq+r)(\alpha^2+z)}{\det(N)}=y_2.$$

\section{Conclusions}
The problem, connected to the Hermite problem, of finding a periodic writing for cubic irrationalities has been solved, providing a periodic representation via ternary continued fraction with rational partial quotients. A periodic representation involving integer numbers can be directly derived from it. The periodic ternary continued fraction can be obtained from a modification of the classical Jacobi algorithm. In particular a family of algorithms based on Eqs. \eqref{fzgz} can be derived. These algorithms become periodic when the input is a couple $(\cfrac{r}{\alpha},\alpha)$, for any cubic irrationality $\alpha$. The ternary continued fraction \eqref{hermite} is very manageable, since the pre--period has length 2, the period has length 3, and it provides simultaneous rational approximations. Of course, many questions and further developments remain open:
\begin{itemize}
\item Generalization of Theorem \ref{main} to any algebraic irrationality is the natural development of the present work. 
\item The problem of the periodicity of the original Jacobi algorithm is still open. It would be possible to use the present work in order to solve this question. Indeed, the sequence of numerator and denominator of the convergents of \eqref{hermite} are linear recurrent sequences. Thus, it could be possible to prove that a ternary continued fraction provided by the Jacobi algorithm and equal to \eqref{hermite} (i.e. converging to the same couple of cubic irrationalities) has numerators and denominators of the convergents that are linear recurrent sequences. Moreover, it could be possible to generalize to ternary continued fraction the result of Lenstra and Shallit \cite{Lenstra}, which proved that a continued fraction is periodic if and only if numerators (or denominators) of the convergents are linear recurrent sequences.
\item The convergence's rate of the ternary continued fraction \eqref{hermite} has not been studied in the present paper and it will be developed in future works, studying the role of the parameter $z$.
\item Cerruti polynomials \eqref{cerruti} appear to be very interesting and they could be applied in different fields of number theory. Indeed, they are a generalization of the R\'edei rational functions. Since R\'edei rational functions are very useful in several fields of number theory, Cerruti polynomials could have many different applications. 

\end{itemize}

%\section{Acknowledgment}
%I would like to thank Prof. Cerruti for the continous support of this work and for introducing me to the beauty of mathematics and its problems, like the Hermite problem.\\
%I would like to thank Prof. Ferrarese and Prof. Roggero for their helpful suggestions.\\
%I am grateful to my friends Marianna, Roberta, Simone, and Vanni for the helpful discussions on this work and to my family.

\end{document}